\documentclass[a4paper,10pt]{amsart}
\usepackage{tikz}
\usepackage{todonotes}

\usepackage{amsmath}
\usepackage[all]{xy}
\usepackage{amsfonts}
\usepackage{amssymb}
\usepackage{graphicx}
\usepackage{epstopdf}
\usepackage{longtable}
\usepackage{hyperref}

\newcommand{\myhline}{\noalign{\global\arrayrulewidth.95pt}\hline
                      \noalign{\global\arrayrulewidth.2pt}}

\theoremstyle {plain}
\def\aa{\ensuremath{\mathfrak{a}}}

\def\TT{\mathbb{T}}
\def\FF{\mathcal{F}}
\def\SSS{\mathcal{S}}
\def\DDD{\mathcal{D}}
\def\FFs{\mathcal{F}_\square}
\def\FFt{\mathcal{F}_\lhd}
\def\KK{\mathbb{K}}
\def\LL{\mathbb{L}}
\def\ZZ{\mathbb{Z}}

\def\ZZZ{\ZZ_{\geq 0}}

\def\KT#1{\KK[T_1,\ldots,T_{#1}]}
\def\LTT#1{\LL[T_1,\ldots,T_{#1}]}
\def\<{\langle}
\def\>{\rangle}

\def\b{\overline}
\def\Dom{{\rm Dom}}

\def\LT{{\rm LT}}
\def\red{{\rm red}}

\def\LC{{\rm LC}}

\newcommand{\gdw}{\Longleftrightarrow}

\newtheorem{theorem}{Theorem}[section]
\newtheorem{lemma}[theorem]{Lemma}
\newtheorem{proposition}[theorem]{Proposition}

\theoremstyle{definition}
\newtheorem{definition}[theorem]{Definition}
\newtheorem{example}[theorem]{Example}
\newtheorem{algorithm}[theorem]{Algorithm}
\newtheorem{construction}[theorem]{Construction}
\newtheorem{remark}[theorem]{Remark}
\theoremstyle {remark}


\renewcommand{\phi}{\varphi}
\renewcommand{\epsilon}{\varepsilon}
\def\eps{\epsilon}
\def\qedh{\belowdisplayskip=-1.2em}

\begin{document}

\title[Monomial containment test]{A test for monomial containment}

\author[S.~Keicher and T.~Kremer]{Simon~Keicher and Thomas~Kremer}

\address{Mathematisches Institut\\
Universit\"at T\"ubingen\\
Auf der Morgenstelle 10\\
72076 T\"ubingen\\
Germany}
\email{keicher\textcircled{$\alpha$}mail.mathematik.uni-tuebingen.de}

\address{Mathematisches Institut\\
Universit\"at T\"ubingen\\
Auf der Morgenstelle 10\\
72076 T\"ubingen\\
Germany}
\email{kremer\textcircled{$\alpha$}mail.mathematik.uni-tuebingen.de}

\thanks{The first author was supported by the DFG Priority Program SPP 1489.}

\begin{abstract}
We present an algorithm 
to decide whether a given ideal in the polynomial
ring contains a monomial
without using Gr\"obner bases, factorization
or sub-resultant computations.
\end{abstract}
\maketitle

\section{Introduction}

Let $\KK$ be a field.
Given an ideal 
$I\subseteq \KT{r}$, 
the {\em monomial containment problem\/}
is to decide whether $I$ contains a monomial.
Equivalently, one is interested in
whether the intersection 
$V(I)\cap \TT^r$ 
of the zero set $V(I)\subseteq \b\KK^r$ with the
algebraic torus 
$\TT^r := (\b\KK^*)^r$
is empty.
The monomial containment problem 
occurs frequently when
determining tropical varieties~\cite{BoJeSpeStu}
or when determining GIT-fans~\cite{Ke}.
The usual approach is via Gr\"obner bases: 
$I$ contains a monomial if and only if 
the saturation $I:(T_1\cdots T_r)^\infty$
contains $1\in \KT{r}$.
This can also be decided by a 
radical membership test:
 $I$ contains a monomial if and only if 
 $T_1\cdots T_r\in \sqrt{I}$.

 In the present
paper, we provide a direct approach involving
neither Gr\"obner basis computations nor (sub-)resultants
or factorization of polynomials. 
We consider more generally the following problem:
given a polynomial $g\in \KT{r}$,
prove or disprove
the existence of an element  $x\in \b\KK^r$ such that
\begin{align}
 \label{eq:solution}
  f(x)\ =\ 0\quad
  \text{for all }
  f\in I,\qquad
  g(x) \ne\ 0.
\end{align}
Clearly,
setting
$g:= T_1\cdots T_r\in \KT{r}$
in \eqref{eq:solution},
the existence of such $x$ 
is equivalent to the monomial containment problem.
Our algorithm, Algorithm~\ref{algo:containsmonomial},
proceeds in three steps:\label{list:intro}
\begin{enumerate}

\item 
Compute finite subsets $S_1,\ldots,S_m\subseteq \KT{r}$
that are in {\em triangular shape\/} and polynomials
$g_1,\ldots,g_m$
such that the solutions of~\eqref{eq:solution} are preserved,
i.e.,
the zero sets satisfy
$$
V(I)\setminus V(g)
\ \ =\ \
\bigcup V(S_i)\setminus V(g_i)
\ \ \subseteq\ \ \b\KK^r.
$$

\item
Making certain variables $T_j$ invertible,
we obtain a function field $\LL$
and an embedding 
$\iota\colon \KT{r}\to \LL[T_{k_1},\ldots,T_{k_s}]$
such that the embedded equations~$\iota(S_i)$
are {\em dense\/}, i.e.,
each variable $T_{k_j}$ corresponds to an equation.

\item 
Then an element $x\in \b\KK^r$ satisfying~\eqref{eq:solution}
exists
if and only if 
the minimal polynomial of the class
$\b{\iota(g_i)} \in \LL[T_{k_1},\ldots,T_{k_s}]/\< S_i \>$
is not a monomial
for some~$i$.
\end{enumerate}
Experiments with our implementation of Algorithm~\ref{algo:containsmonomial}
suggest that it is competitive for certain classes of input; 
for instance,
it usually beats the Gr\"obner basis approach when 
a solution exists, i.e., the ideal is monomial-free.

Note that the idea behind step (i) of the algorithm
is quite common and similar concepts have been used
by several authors for a more explicit study or even
the explicit computation of solutions.
See, e.g.,~\cite{AuLaMoMa,AuMoma,Chen,Ritt,Wu}
for a series of papers with  Gr\"obner basis-free algorithms for systems of equations.
The methods of Wang~\cite{Wa}, Thomas~\cite{thomas,thomas2}
 as well as B\"achler, Gerdt, Lange-Hegermann and Robertz~\cite{BaeGeLaRo} 
can also deal with systems of equations and inequalities.
 They determine the solutions of such
 systems 
 by means of certain triangular sets called 
 {\em simple systems\/}; 
 their computation involves sub-resultant computations.
All algorithms, including ours in step (i),
share the concept of {\em triangular sets\/},
  certain finite subsets $S_i\subseteq \KT{r}$
  such that $V(I)= \bigcup V(S_i)$ holds.
  The $S_i$ then give insight into the structure of
  the solution set 
 $V(I)\subseteq \b\KK^r$.
  As we are only interested in solvability of~\eqref{eq:solution},
 we will only need triangular sets with 
 weaker properties but which can be computed more efficiently.

 The structure of this paper is as follows.
 In Section~\ref{sec:algotriag}, we show how 
 to decompose the given ideal into a list of triangular sets
 with sufficient properties for our solvability test; this is step (i)
 in the previous list.
 Section~\ref{sec:algomonomial} is devoted to steps (ii) and (iii),
 i.e., we show how to reduce the problem to a dense system over a function field
  and how to determine the solvability of such a system by 
 means of minimal polynomial computations.
  Explicit algorithms are given in each section.
  In Section~\ref{sec:implem}, we present our algorithm for the monomial
  containment problem. We compare the
 experimental running time of the \texttt{perl} 
implementation~\cite{monomtest:implem}
 of the algorithm to the Gr\"obner basis approach
 as well as to the methods of~\cite{BaeGeLaRo,Wa}.

This paper builds on~\cite{kremer}.
 We would like to thank J\"urgen Hausen for helpful discussions.

\section{Triangular shape}
\label{sec:algotriag}

In this section, we treat item (i) of the list on page~\pageref{list:intro},
i.e., we decompose a system as in~\eqref{eq:solution} with an ideal
 $I\subseteq \KT{r}$ and a polynomial $g\in \KT{r}$ into a list of finite sets
of polynomials that are in {\em triangular shape\/}.
We show how to compute this decomposition by iteratively applying a set of 
operations that do not change the solvability of~\eqref{eq:solution}.

We first define the notion of {\em triangular shape\/}.
In the literature, they are also called
{\em triangular sets\/}~\cite{AuLaMoMa,AuMoma, DeLo, GrePfi}.

\begin{definition}
\label{def:triag}
Fix the lexicographical ordering 
$T_1>\ldots>T_r$
on $\KT{r}$.
We call polynomials $f_1,\ldots,f_s\in \KT{r}$
of {\em triangular shape\/}
if  for each $f_j$, 
there is $1\leq k(f_j)\leq r$ such that 
\begin{enumerate}
\item
we have
$k(f_1)<\ldots <k(f_s)$,

\item 
$f_j\in \KK[T_{k(f_j)},\ldots, T_r]\setminus\KK[T_{k(f_j)+1},\ldots,T_r]$
holds for each $1\leq j\leq s$.

\end{enumerate}
We denote by $\deg_{T_i}(f)$ the {\em ($T_i$-)degree\/}
of a polynomial $f\in \KT{r}$
considered as an element of the univariate polynomial ring
$\KK[T_j;\ j\ne i][T_i]$.
Moreover, we write
$$
\LC_{k(f_i)}(f_i)
\ \ \in\ \ 
R_{<k(f_i)}
\ :=\ 
\KK\left[T_{k(f_i)+1},\ldots,T_r\right].
$$
for the leading coefficient of the polynomial $f_i$
considered in the ring $R_{<k(f_i)}[T_{k(f_i)}]$.
\end{definition}

We now introduce the concept of {\em (semi-) triangular systems\/}.
Assume $I$ is generated by polynomials $f_1,\ldots,f_s\in \KT{r}$.
We sort them into two sets (and keep track of the inequality~$g$):
polynomials that are already in triangular shape $\FFt$
and remaining polynomials $\FFs$.

\begin{definition}
A {\em semi-triangular system\/} (of equations)
is a tuple 
$(\mathcal F_\square, \FFt,k,g)$
consisting of finite subsets 
$\mathcal F_\square$, $\FFt\subseteq \KT{r}$,
an integer $0\leq k\leq r$ and a polynomial $g\in \KT{r}$ such that
\begin{enumerate}
\item 
$\FFt$ is of triangular shape,

\item
we have
$\LC_{k(f)}(f)\mid g$ for all $f\in \FFt$,

\item
the set
$\{1,\ldots,k\}$ contains
$\{k(f);\ f\in \FFt\}$,

\item
for all
$f\in \mathcal F_\square$ 
and each $1\leq i\leq k$
we have
$\deg_{T_i}(f)=0$.

\end{enumerate}
Moreover, we call a semi-triangular system
$(\mathcal F_\square, \FFt,k,g)$
a {\em triangular system\/}
if $\mathcal F_\square \subseteq \KK$ holds.
\end{definition}

\begin{example}
\label{ex:triag}
Define in $\KK[T_1,T_2,T_{3}]$
the subsets
$\mathcal F_\square := \emptyset$
and 
$\FFt := \{f_1,f_2,f_3\}$
where the $f_i$
and $k(f_i)$ are 
\begingroup
\footnotesize
$$
\begin{array}{rcrrcl}
f_1
&:=&
T_1^2 - (T_2 + T_3)T_1,
&
\qquad
k(f_1)
&=&
1,
\\
f_2
&:=&
T_2^2- T_3\qquad\ ,
&
k(f_2)
&=&
2,
\\
f_3
&:=&
T_3^2 - T_3,
&
k(f_3)
&=&
3.
\end{array}
$$
\endgroup
Then $\FFt$ is of triangular shape
and 
$(\mathcal F_\square, \FFt, 3, T_1T_2T_3)$
is a triangular system.
\end{example}

\begin{definition}
A list $\SSS$ of semi-triangular systems 
is called a {\em triangle mush\/}.
Two triangle mushes $\SSS$ and $\SSS'$
are {\em equivalent\/} if we have $V(\SSS) = V(\SSS')$
with the {\em solutions\/} 
\[
V(\SSS)
\ :=\! 
\bigcup_{(\mathcal F_\square, \FFt,k,g)\in\SSS}
\!
V\left(
\mathcal F_\square \cup \FFt
\right)
\setminus
V(g)
\ \ \subseteq\ \
\b\KK^r.
\]
For the case of a single element $\SSS = \{S\}$, 
we will use the same notions
for $S$ instead of $\SSS$.
\end{definition}

\begin{example}
\label{ex:mush}
Consider the triangle mush 
$\SSS := \{(\FFs, \emptyset, 0, g)\}$ in $\KT{4}$
where $g:=T_1T_2T_3$ and $\FFs$ consists of the two polynomials
\[
f_1
\ :=\ 
(T_3-T_1)(T_3-T_2)T_2
,
\qquad
f_2
\ :=\ 
(T_1+T_2-T_3)T_4.
\]
Going through the different cases,
one directly verifies that
$V(\SSS)\subseteq \b\KK^4$
consists of all points
$(x_1,x_2,x_1,0)$
and 
$(x_1,x_2,x_2,0)\in \b\KK^4$
where $x_i\in \b\KK^*$.
We will continue this example in~\ref{ex:containsmon}.
\end{example}

Given a triangle mush $\SSS$,
we are interested in operations that 
transform $\SSS$ into an equivalent
triangle mush $\SSS'$ that consists 
of triangular systems.

\begin{construction}[Solution-preserving operations]
\label{con:mushops}
Let $\SSS := \{(\mathcal F_\square, \FFt,k,g)\}$ 
consist of a semi-triangular system.
Each of the following operations produces an equivalent triangle mush
$\SSS'$.
\begin{enumerate}

\item 
{\em Case-by-case analysis\/}:
If $f\in \KK[T_{k+1},\ldots,T_r]$ and $h\in \KT{r}$ are such that
$g\mid h$ and $h\mid fg$, then one may choose
\[
\SSS'
\ :=\ 
\left\{
(\mathcal F_\square\cup\{f\}, \FFt,k,g),\,
(\mathcal F_\square, \FFt,k,h)
\right\}.
\]

\item 
{\em Polynomial division\/}:
Consider $f,h\in \FFs$ and $b\in \KT{r}$ 
with $b\mid g$.
Assume that for some $j\in \ZZZ$ 
we have
\[
b^jf\ =\ ah + u,\qquad
a,u\ \in\ \KK[T_{k+1},\ldots,T_r]
\]
where 
$b := \LC_{T_{k+1}}(h)$
and 
$\deg_{T_{k+1}}(u)<\deg_{T_{k+1}}(h)$.
Then we choose the triangle mush
\[
\SSS'\ :=\ 
\{
(\mathcal F_\square\setminus\{f\}\cup \{u\}, \FFt,k,g)
\}.
\]

\item 
{\em Unused variable\/}:
If $k<r$ and $\deg_{T_{k+1}}(f)=0$ holds for each $f\in \FFs$,
then we may choose
\[
\SSS'\ :=\ 
\{
(\mathcal F_\square, \FFt,k+1,g)
\}.
\]

\item 
{\em Sort polynomial\/}:
If $k<r$ holds and there is exactly one polynomial 
$f\in \FFs$ with $\deg_{T_{k+1}}(f)\ne 0$
and $\LC_{k(f)}(f)\mid g$,
then we may choose
\[
\SSS'\ :=\ 
\left\{
(\mathcal F_\square\setminus\{f\}, \FFt\cup \{f\},k+1,g)
\right\}.
\]

\item 
{\em Last polynomial\/}:
Assume $k<r$ and there is exactly one polynomial 
$f\in \FFs$ with $\deg_{T_{k+1}}(f)\ne 0$.
For $-1\leq j\leq d$, we write
\begin{eqnarray*}
f
&=&
\sum_{i=0}^d a_iT_{k+1}^i
,\qquad
f_j
\ \ :=\ \ 
\sum_{i=0}^j a_iT_{k+1}^i
\ \in\ R_{<k+1}[T_{k+1}],
\\
\FFt^j
&:=&
\FFt
\cup
\{f_j\},
\qquad
\FFs^j
\ \ :=\ \
\left(
\FFs
\setminus
\{f\}
\right)
\cup
\{a_{j+1},\ldots,a_d\}.
\end{eqnarray*}
Then we may choose
\begin{gather*}
\SSS'\ :=\ 
\bigl\{
(\mathcal F_\square^1, \FFt^1,k+1,ga_1),
\ldots,
(\mathcal F_\square^d, \FFt^d,k+1,ga_d),
\\
(\mathcal F_\square^{-1}, \FFt,k+1,g)
\bigr\}.
\end{gather*}
\end{enumerate}
\end{construction}

\begin{proof}
One directly checks that in all cases $\SSS'$ 
is a triangle mush.
For (i),
each
$x\in V(\SSS)$
either satisfies $f(x)=0$ and $g(x)\ne 0$
or we have $f(x)\ne 0$ and $h\mid fg$ implies 
$h(x)\ne 0$,
i.e., $x\in V(\SSS')$.
The inclusion 
$V(\SSS')\subseteq V(\SSS)$
is clear from $g\mid h$.
We come to (ii). Each $x\in V(\SSS)$
satisfies
\[
 u(x)
 \ \ =\ \ 
 b(x)^jf(x) - a(x)h(x)
 \ \ =\ \ 
 0.
\]
For the reverse inclusion,
we use $b\mid g$ to obtain $b(x)\ne 0$.
Consequently, we may infer $f(x)=0$ from
\[
 b(x)^jf(x)
 \ \ =\ \ 
 (b^jf)(x)
 \ \ =\ \ 
 a(x)h(x) + u(x)
 \ \ =\ \ 
 0.
\]
Operations (iii) and (iv) are clear.
For (v), we define the following triangle mushes
for $0\leq l\leq d$:
\[
 \SSS_l\, :=\, 
 \{
 (\FFs^l\cup \{f_l\}, \FFt,k,g)
 \},
 \qquad
 \DDD_l\, :=\, 
 \{
 (\FFs^j,\FFt^j,k+1,ga_j);\ l<j\leq d)
 \}.
\]
Observe that by an application of operation (i), we obtain
an equality of solutions
\[
V(\SSS_l)
\ \ =\ \ 
V\left(
\left\{ 
\left(\FFs^{l-1}\cup \{f_l\},\FFt,k,g\right),\ 
\left(\FFs^l\cup \{f_l\},\FFt,k,ga_l\right)
\right\}
\right).
\]
As the ideal
$\<\FFs^{l-1}\cup \{f_l\}\>$
equals
$\<\FFs^{l-1}\cup \{f_{l-1}\}\>$
and by an application of operation (iv),
we obtain
\begin{eqnarray*}
V(\SSS_l)
&=& 
V\left(
\left\{ 
\left(\FFs^{l-1}\cup \{f_{l-1}\},\FFt,k,g\right),\ 
\left(\FFs^l,\FFt^l,k+1,ga_l\right)
\right\}
\right)
\\
&=& 
V\left(\SSS_{l-1}\cup  (\DDD_{l-1}\setminus\DDD_l)\right).
\end{eqnarray*}
Adding the equations stored in $\DDD_l$ 
on both sides does not change
the solution set, i.e.,
$V\left(\SSS_{l}\cup \DDD_{l}\right)$
is equal to
$V\left(\SSS_{l-1}\cup \DDD_{l-1}\right)$.
Iteratively, we obtain
$V\left(\SSS_{d}\cup \DDD_{d}\right)
=V\left(\SSS_{0}\cup \DDD_{0}\right)$.
Moreover, because of $f_0 = a_0$ and operation (iii):
$$
V(\SSS_0)
\ \ =\ \ 
V\left((\FFs^{-1},\FFt,k,g)\right)
\ \ =\ \ 
V\left((\FFs^{-1},\FFt,k+1,g)\right).
$$
We conclude 
that $V(\SSS)$ equals
$V(\SSS_d\cup \DDD_d)= 
V(\SSS_0\cup \DDD_0)$
which in turn is the same
as the solution set~$V(\SSS')$.
\end{proof}

The next algorithm transforms a triangle mush
into an equivalent triangle mush consisting only 
of triangular systems.
Given a triangular system
$(\FFs,\FFt,k,g)$,
the idea is to reduce $T_{k+1}$-degrees of an element
$f$ of the unsorted polynomials $\FFs$
by successive polynomial divisions;
afterwards, we move $f$ into the set of sorted polynomials~$\FFt$.

Given a finite set of polynomials 
$\FF\subseteq \KT{r}$,
its {\em reduction\/}  is
a finite subset 
$\red(\FF)\subseteq \KT{r}$ such that
\begin{gather*}
\LT(f_1)\ \nmid\ \LT(f_2)
\quad\text{for all }
f_1,f_2\in \red(\FFs),
\\
\<\LT(\FF)\>\ \subseteq\ \<\red(\LT(\FF))\>,
\qquad\qquad
\<\FF\>\ =\ \<\red(\FF)\>
\end{gather*}
where we denote by $\LT(f)$ or $\LT(M)$ the leading term of a polynomial
$f$ or set of polynomials $M$ with respect to the ordering
defined in Section~\ref{sec:algotriag}.
Computing the reduction of $\FF$ means successively
applying the division algorithm to the elements of $\FF$, 
see, e.g.,~\cite{CoLiOSh}.

\begin{algorithm}[MakeTriangular]
\label{algo:triang}
{\em Input: } 
a triangle mush $\SSS$ in  $\KT{r}$.
\begin{itemize}
\item 
While there is $S:=(\mathcal F_\square, \FFt,k,g)\in\SSS$
with $k<r$, do:
  \begin{itemize}
    \item
    Replace $\mathcal F_\square$ by its reduction $\red(\mathcal F_\square)$.
    \item 
    If there is $f\in \mathcal F_\square$ 
    with $\deg_{T_{k+1}}(f)>0$, then:
    \begin{itemize}
      \item 
      If there is $h\in \mathcal F_\square\setminus\{f\}$ 
      with $\deg_{T_{k+1}}(h)>0$, then:
      \begin{itemize}
	\item 
	Perform a polynomial division of $f$ by $h$ in 
	the univariate polynomial ring
	$R:=\KK(T_{k+2},\ldots,T_r)[T_{k+1}]$
	to obtain
	$$
	f \ =\  a'h + u'\ \in\ R.
	$$ 
	\item
	Set $b:=\LC_{k+1}(h)\in \KK[T_{k+2},\ldots,T_r]$ 
	and $j:=\deg_{T_{k+1}}(h)+1\in \ZZZ$.
	With $a :=b^ja'$ and $u:=b^ju' \in \KK[T_{k+1},\ldots,T_r]$
	we then have
	$$
	\qquad\qquad\qquad
	b^jf\ =\ ah +u\ \in\ \KK[T_{k+1},\ldots,T_r].
	$$
	\item
	Redefine 
      $\SSS := (\SSS\setminus \{S\}) \cup \{S',S''\}$
      where 
      \begin{eqnarray*}
      \qquad\qquad\qquad
      S' 
      &:=&
      (\mathcal F_\square\setminus\{f\}\cup \{u\}, \FFt,k,bg),\\
      S'' 
      &:=&
      (\mathcal F_\square\cup\{b\}, \FFt,k,g).
      \end{eqnarray*}
      \end{itemize}
      \item 
      Otherwise, if there is no such $h$, then:
      \begin{itemize}
      \item  
      Redefine 
      $\SSS := (\SSS\setminus \{S\}) \cup \{S',S_{1},\ldots,S_{d}\}$
      where with the notation of Construction~\ref{con:mushops}~(v):
      \begin{eqnarray*}
     \qquad\qquad\qquad
     S' &:=& \left(\mathcal F_\square^{-1}, \FFt,k+1,g\right),
      \\
      S_j &:=& \left(\mathcal F_\square^{j}, \FFt^{j},k+1,ga_j\right).     
      \end{eqnarray*}
      \end{itemize}
    \end{itemize}
    \item Otherwise, if there is no such $f$, then:
    \begin{itemize}
      \item 
      Redefine 
      $\SSS := (\SSS\setminus \{S\}) \cup \{S'\}$
      where 
      $S' := (\mathcal F_\square, \FFt,k+1,g)$.
    \end{itemize}
  \end{itemize}
  \item
  Define $\SSS' := \SSS$.
\end{itemize}
{\em Output: }
$\SSS'$.
Then $\SSS'$ is a triangle mush 
that is equivalent to $\SSS$
and consists of triangular systems.
\end{algorithm}

\begin{proof}
Note that we use only operations described 
in Construction~\ref{con:mushops};
for instance, the replacement of $\SSS$
by $(\SSS\setminus \{S\}) \cup \{S',S''\}$
is an application of, first, operation (i) and then operation (ii).
Therefore,
$\SSS'$ is equivalent to~$\SSS$.
As each $S:=(\FFs,\FFt,k,g)\in \SSS'$ satisfies $k=r$,
each element of $\FFs$ is constant, i.e., $S$ is 
triangular.

It remains to show that Algorithm~\ref{algo:triang}
terminates.
To this end, consider the infinite digraph
$G'=(V',E')$ where $V'$ is the set of all semi-triangular
systems over $\KT{r}$ and, given vertices $S_1,S_2\in V$,
the edge $(S_1,S_2)\in E'$ exists
if and only if Algorithm~\ref{algo:triang}
replaces $S_1$ within a single iteration of
the while-loop
by a triangle mush 
$\SSS'$ with $S_2\in  \SSS'$.
Let $G=(V,E)$ be the subgraph induced by all 
semi-triangular systems that are reachable
by a path starting in~$\SSS$.

Consider a path $(S_1,S_2,\ldots)$ in $G$, i.e., $S_i\in V$
and $(S_i,S_{i+1})\in E$ for all~$i$.
We write $S_i = (\FFs^i,\FFt^i,k_i,g_i)$.
By construction, $k_i\leq k_{i+1}\leq r$ holds for all $i$.
This means there is $i_1\in \ZZ_{\geq 1}$ such that 
$k_{i+1}=k_i$ for all $i\geq i_1$
and Algorithm~\ref{algo:triang} will perform the polynomial division
$b^jf=ah+u$,
i.e., operation (ii)
of Construction~\ref{con:mushops}, 
for each such~$S_i$.
Since always
$\deg_{T_{k_i+1}}(b)=0$ holds,
we have
$\deg_{T_{k_i+1}}(f)>\deg_{T_{k_i+1}}(u)$
and the reduction step only reduces $T_{k_i+1}$-degrees,
the sequence
\[
(N_i)_{i\geq i_1},\qquad
N_i\ :=\ \sum_{f\in \FFs^i}\deg_{T_{k_i+1}}(f)\ \in\ \ZZZ
\]
is monotonically decreasing.
As $N_i \in \ZZZ$ holds, 
this sequence either is finite or 
becomes stationary.
Assume the latter holds, i.e., there is 
$i_2\in \ZZ_{\geq i_1}$
such that
$N_i = N_{i+1}$ is valid for all $i\geq i_2$.
This implies, that for all $i\geq i_2$
in the polynomial division step
only the ``$b$-part'' will be added, i.e.,
$$
\FFs^{i+1}\ =\ \FFs^i \cup \{b\}.
$$
In particular, the ideal $\<\LT(\FFs^i)\>$ is contained in 
$\<\LT(\FFs^{i+1})\>$ for each $i\geq i_2$.
As $\KT{r}$ is noetherian, the chain
$$
\left\<\LT\left(\FFs^{i_2}\right)\right\>
\ \subseteq\ 
\left\<\LT\left(\FFs^{i_2+1}\right)\right\>
\ \subseteq\ 
\ldots
$$
becomes stationary, i.e., there is $i_3\in \ZZ_{\geq 1}$
such that 
$\<\LT(\FFs^i)\>=\<\LT(\FFs^{i+1})\>$
holds for all $i\geq i_3$.
Moreover, as 
$b = \LC_{k+1}(h)$
holds and 
$h\in \red(\FFs^i)$, 
we have
\[
\LT(b)
\ \notin\ 
\left\<\LT\left(\red\left(\FFs^i\right)\right)\right\>
\ \supseteq\  
\left\<\LT\left(\FFs^i\right)\right\>.
\]
Then $b$ cannot be an element of 
$\FFs^{i+1}$ for $i\geq i_3$, a contradiction.
Thus, the sequence $(N_i)_i$ is finite.
In turn, this forces the $(S_1,S_2,\ldots)$ to be finite and acyclic.

Since each vertex $S\in V$ is adjacent to only 
finitely many vertices, the previous
argument shows that
$G$ is a finite tree.
In particular, the while-loop in Algorithm~\ref{algo:triang}
will be executed at most $|G|$ times for each vertex $S\in V$,
i.e., the algorithm terminates.
\end{proof}

\begin{remark}
\label{rem:simplesystems}
 Algorithm~\ref{algo:triang}
 is similar to the decomposition 
 into {\em simple systems\/}
 used in~\cite{BaeGeLaRo}.
 Note, however, that they are interested in special properties
 (e.g., disjointness) of this decomposition
 whereas ours is weaker but needs not use 
 operations like $\gcd$ or subresultant computations.
 \end{remark}

An example computation with
Algorithm~\ref{algo:triang}
will be performed at the end of the next section
in Example~\ref{ex:containsmon}.

\section{Solvability}
\label{sec:algomonomial}

We now come to steps (ii) and (iii) in the list on page~\pageref{list:intro}:
as before, we assume we are given an ideal
$I=\<f_1,\ldots,f_s\>\subseteq\KT{r}$
and a polynomial $g\in \KT{r}$
and want to answer the question whether there is $x\in \b\KK^r$
satisfying~\eqref{eq:solution}.

Using Algorithm~\ref{algo:triang} of the previous section
with input $I$ and $g$, we
obtain an equivalent
triangle mush $\SSS$ that consists of triangular
systems. 
Note that we can replace each system $(\FFs, \FFt, k, g)\in \SSS$ 
with $\FFs = \{0\}$ by the equivalent system 
$(\emptyset, \FFt, k, g)$;
systems with $\FFs \cap \KK^*\ne \emptyset$
clearly are not solvable.
Then~\eqref{eq:solution} 
can be rephrased as the question,
whether there is $x\in \b\KK^r$ such that
\[
f(x)\ =\ 0
\qquad
\text{for all}\ 
f\in \FFt
,\qquad
g(x)\ \ne \ 0
\]
holds for some $(\emptyset, \FFt, k, g)\in \SSS$.
Consequently, it suffices to present methods for the case
$\SSS=\{S\}$ of a single triangular system.
Here is an overview of the steps to test
whether $V(S)\ne \emptyset$ holds:
\begingroup
\small
\[
 \xymatrix@R=.1pt@C=8mm{
 \KT{r}
 &
 \LL[T_{k_1},\ldots,T_{k_s}]
 &
 \LL[T_{k_1},\ldots,T_{k_s}]
 &
 \LL[T_{k_1},\ldots,T_{k_s}]/\<\FFt'\>
  \\ 
\text{\rotatebox{90}{$\subseteq$}}
 &
\text{\rotatebox{90}{$\subseteq$}}
 &
\text{\rotatebox{90}{$\subseteq$}}
 &
 \text{\rotatebox{90}{$\in$}}
 \\ 
 S
 \ar@{|->}[r]^{\ref{prop:fieldchange}}
 &
 \iota(S)
 \ar@{|->}[r]^(.4){\ref{prop:normalize}}
 &
 (\emptyset, \FFt', k',g')
 \ar@{|->}[r]^(.6){\ref{prop:solvable}}
 &
 \b{g'}
 \\ 
 &
\text{\tiny dense}
 &
 \text{\tiny dense, monic}
 &
 \text{\tiny min.~polyn.~monomial?}
 }
\]
\endgroup
Here, $\LL$ is a suitable function field.
The following proposition reduces the treatment
of a triangular system 
in $\KT{r}$
to a triangular, {\em dense\/} 
system in $\LL[T_{k_1},\ldots,T_{k_s}]$,
i.e.,
a triangular system 
$(\emptyset, \{f_1,\ldots,f_s\}, k,g)$ 
such that the set
$\{k_1,\ldots,k_s\}$
coincides with
$\{k(f_1),\ldots,k(f_s)\}$.

\begin{proposition}[Swap the field]
\label{prop:fieldchange}
Consider a triangular system
$S := (\emptyset, \FFt, k, g)$
in $\KT{r}$.
Write $\FFt = \{f_1,\ldots,f_s\}$
and let $k_i := k(f_i)\in \ZZ_{\geq 1}$
be as in Definition~\ref{def:triag}.
Under the canonical embedding
\[
\iota\colon
\KT{r}\, \to\, 
\LL[T_{k_1},\ldots,T_{k_s}],\qquad
\LL\, :=\, 
\KK\left(T_{i};\ i\not\in\{k_1,\ldots,k_s\}\right)
\]
we obtain a triangular system 
$\iota(S) := (\emptyset, \iota(\FFt),s,\iota(g))$
that is dense 
in the polynomial ring 
$\LL[T_{k_1},\ldots,T_{k_s}]$.
Moreover, we have
\[
V(S)\ \ne \ \emptyset
\qquad\gdw\qquad
V(\iota(S))\ \ne\ \emptyset.
\]
\end{proposition}

For the proof of Proposition~\ref{prop:fieldchange}
we recall from~\cite[Ch.~VI]{ZaSa}
the generalization of evaluation homomorphisms;
we will need this to control the elements
in $\b\LL$.
A {\em place\/}
is a $\b\KK$-homomorphism
$\eps\colon R_\phi\to \overline{\KK}$
with a subring $R_\eps\subseteq \overline\LL$
such that 
$$
x\,\in\,\overline\LL\setminus R_\eps
\quad\Longrightarrow\quad
x^{-1}\in R_\eps
\ \text{ and }\ 
\eps(x^{-1})\,=\,0.
$$

Given $x\in \overline \KK^{r-s}$, 
denote by 
$\epsilon_x'\colon \b\KK[T_i;\ i\notin \{k_1,\ldots,k_s\}]\to \overline \KK$
the evaluation homomorphism.
According to~\cite[Thm.~5 in~VI.4]{ZaSa},
we have
\[
\xymatrix{
\b\KK[T_i;\ i\notin \{k_1,\ldots,k_s\}]
\ar@{}|(.72)\subseteq[r]
\ar[dr]_{\eps_{x}'}
&
R_{\eps_x}
\ar@{}|(.5)\subseteq[r]
\ar[d]^{\eps_x}
&
\b\LL
\\
&
\b\KK
&
}
\]
with a place 
$\epsilon_x\colon R_{\epsilon_x}\to\overline \KK$
extending $\eps_x'$.
Moreover, we define the {\em domain\/} of  $t=(t_1,\ldots,t_s)\in \overline \LL^s$ as the intersection
\[
\Dom(t)\ :=\ 
\bigcap_{i=1}^s\Dom(t_i)
,\qquad
\Dom(t_i)\ :=\ 
\left\{
y\in \overline \KK^{r-s};\ t_i\in R_{\epsilon_y}
\right\}.
\]

\begin{lemma}
\label{lem:place}
In the situation of Proposition~\ref{prop:fieldchange},
assume we have $k_1 =1,\ldots,k_s =s$. 
Then the following claims hold.
\begin{enumerate}
\item
Consider $x\in \b\KK^{r-s}$ and $t_1,\ldots,t_n\in \b\LL$
satisfying $\eps_x(t_1\cdots t_n) = 0$.
Then there is $1\leq j\leq n$ such that $\eps_x(t_j)=0$.

\item 
For each 
$t\in V(\iota(S))\subseteq \b\LL^s$
and each 
$x\in \Dom(t)\subseteq \overline \KK^{r-s}$, we have
$(\eps_x(t_1),
\ldots,
\eps_x(t_s),x)\in \overline{V(S)}$
where the closure is taken in~$\overline\KK^{r}$.

\item 
Given $x\in V(S)\subseteq \b\KK^r$,
write 
$x=(x'',x')$ with $x'\in \b\KK^{r-s}$, $x''\in \b\KK^{s}$.
Then
there is 
$t\in V(\iota(S))\subseteq \b\LL^s$
such that
$$
x'\,\in\, \Dom(t)\,\subseteq\, \b\KK^{r-s}\qquad\text{and}\qquad
(\eps_{x'}(t_1),\ldots,\eps_{x'}(t_s))\,=\,x''.
$$

\end{enumerate}
\end{lemma}

\begin{proof}

For (i),
we relabel $t_1,\ldots,t_n$ such that there is $k\in \ZZZ$
with $t_i\in R_{\eps_x}$ for all $i\leq k$ and 
$t_i\notin R_{\eps_x}$ for $i>k$.
By definition of places, $\eps_x(t_i^{-1})=0$
for all $i>k$ and thus
\[
\prod_{i=1}^k \eps_x(t_i)
\,=\,
\eps_x\left(
\prod_{i=1}^n t_i
\prod_{i=k+1}^n t_i^{-1}
\right)
\,=\,
\eps_x
\left(\prod_{i=1}^n t_i\right)
\left(\prod_{i=k+1}^n \eps_x\left(t_i^{-1}\right)\right)
\,=\,
0.
\]

For (ii),
given $f\in \sqrt{\<\FFt\>}:g$, we have
$\iota(f)\in \sqrt{\<\iota(\FFt)\>}:\iota(g)$,
which means $\iota(f)(t)=0$.
Write $f = \sum_{\nu} a_{\nu} T^{\nu}$.
From
$$
f(
\eps_x(t_1),
\ldots,
\eps_x(t_s)
,x
)
\ =\ 
\sum_\nu
a_\nu
\prod_{i=1}^{s} \eps_x(t_i)^{\nu_i}
\prod_{j=s+1}^{r} x_{j-s}^{\nu_j}
\ =\ 
\eps_x
\left(
\iota(f)(t)
\right)
\ =\ 
0
$$
we infer that 
$(\eps_x(t_1),
\ldots,
\eps_x(t_s),x)\in \b \KK^r$
is an element of the closure
$\overline{V(S)}=V(\sqrt{\<\FFt\>}:g)$
in~$\b\KK^r$.

We come to (iii). 
We first show by (finite) induction on $0\leq m\leq s$, 
that there are 
$t_{m+1},\ldots,t_s\in \b \LL$
such that for the evaluation 
homomorphism
\[
\theta_m\colon
\KK[T_1,\ldots,T_{r}]\,\to\,
\b\LL[T_1,\ldots,T_m],\qquad
T_j\,\mapsto\,
\begin{cases}
t_j, & m<j\leq s,\\
T_j, & \text{else}
\end{cases}
\]
we have
$\<f_{m+1},\ldots,f_s\>\subseteq \ker(\theta_m)$
and
$\eps_{x'}(t_j) = x_j$ 
holds for each $m<j\leq s$.
Nothing is to prove for $m=s$.
Assume now that this claim holds for a fixed
$1\leq m\leq s$;
we show that it also holds for $m-1$.
Since we have $\LC_{m}(f_i)\mid g$, $g(x)\ne 0$
and $\eps_{x'}(t_j)=x_j$ for $m<j\leq s$,
setting $a := \LC_m(f_m)$,
we obtain
$$
\eps_{x'}
\left(
\theta_m(a)
\right)
\ =\ 
a(x)
 \bigm| 
\eps_x(g)
\ =\ 
g(x)
\ \ne\ 
0.
$$
In particular, $\theta_m(a)\ne 0$.
Therefore, the non-zero univariate polynomial
$f_m':=\theta_m(f_m)\in \b\LL[T_m]$
can be decomposed into linear factors
\[
f_m'
\ =\ 
c\prod_{j=1}^n(T_m - t_{mj})\qquad
\text{with}\ 
t_{mj}\,\in\,\b\LL,\quad
c\,\in\,\b\LL^*.
\]
Note that $c = \theta_m(a)$ holds
and thus $\eps_{x'}(c)\ne 0$.
Moreover, 
using again $\eps_{x'}(t_j)=x_j$ for $j>m$
and $f_m'=\theta_m(f_m)$,
we have
$\eps_{x'}(f_m'(x_m)) = f_m(x)=0$
where the vanishing is due to $x\in V(S)$.
The identity
$$
0
\ =\ 
\eps_{x'}(f_m'(x_m))
\ =\ 
\eps_{x'}
\left(
c\prod_{j=1}^n(x_m - t_{mj})
\right)
$$
together with statement (i)
provide us with $1\leq j\leq n$
such that 
$\eps_{x'}(t_{mj}) = x_m$.
Defining $t_m := t_{mj}$,
the elements
$t_m,\ldots,t_s\in \b\LL$
satisfy the claims:
we have 
$\<f_m,\ldots,f_s\>\subseteq \ker(\theta_{m-1})$
since
$\theta_{m-1}(f_m) = f_m'(t_m)=0$
and
$\eps_{x'}(t_m) = x_m$ holds.

Using this argument, we now have a map
$\theta_0$ such that both
$\<\FFt\>\subseteq \ker(\theta_0)$ and $\eps_{x'}(t_m)=x_m$ hold.
Setting $t := (t_1,\ldots,t_s)$,
we obtain
$$
t\ \in\ V(\FFt)\setminus V(g)\ =\ V(\iota(S))\ \subseteq\ \b\LL
$$
because $f_m(t)=\theta_0(t_m)=0$ for each $1\leq m\leq s$
and
$\eps_{x'}(\theta_0(t))=g(x)\ne 0$
implies in particular
that $\theta_0(t)=g(t)\ne 0$.
By construction, $\eps_{x'}(t)=x''$ holds.
\end{proof}

\begin{proof}[Proof of Proposition~\ref{prop:fieldchange}]
Clearly, the system is dense.
By Lemma~\ref{lem:place}~(iii), $V(S)\ne \emptyset$
implies that also $V(\iota(S))$ is non-empty.
If
for each $t\in V(\iota(S))$, 
there is $x\in \Dom(t)$,
then Lemma~\ref{lem:place}~(ii)
ensures
$\b{V(S)}\ne \emptyset$
and therefore $V(S)\ne \emptyset$.

It thus remains to prove that 
$\Dom(t)\ne \emptyset$.
Let $1\leq j\leq s$ be an integer.
If $t_j = 0$ holds, clearly
$\Dom(t_j)=\overline\KK^{r-s} \setminus V(1)$
is non-empty.
If $t_j\ne 0$, we consider the 
product $f$ of the minimal polynomial of $t_j$
over $\LL$
with its common denominator
and thereby obtain a polynomial $h$:
\begin{align*}
f
\ &=\ 
\sum_{i=0}^{m} a_iX^i,
\quad
h
\ =\ 
\sum_{i=0}^{m}a_{m-i}X^i
\ \ \in\ \KK\left[T_j;\ j\not\in\{k_1,\ldots,k_s\}\right][X]
\end{align*}
where 
$h(t_j^{-1}) = t_j^{-m}f(t_j)=0$.
By definition, each
$x\in \b\KK^{r-s}$ with $x\notin \Dom(t_j)$
must satisfy $\eps_x(t_j^{-1})=0$.
For all $i>0$, 
from $a_{m-i}\in\KK\left[T_j;\ j\not\in\{k_1,\ldots,k_s\}\right]$
we know that $a_{m-i}\in R_{\eps_x}$ holds
and therefore obtain
$\eps_x(a_{m-i}t_j^{-i}) = 0$.
We have
\[
\eps_x(a_m)
\ =\ 
\eps_x\left(
h(t_j^{-1})
-
\sum_{i=1}^m a_{m-i}t_j^{-i}
\right)
\ = \ 
0,
\]
from which we infer that $a_m(x)=0$
and therefore $x\in V(a_m)\subseteq \b\KK^{r-s}$ 
hold; note that the inclusion $V(a_m)\subsetneq \b\KK^{r-s}$
is proper since $a_m\ne 0$.
In other words,
\[
\Dom(t_j)\ \supseteq\ 
\b\KK^{r-s}\setminus V(a_m)
\ \ne\ \emptyset.
\]
As finite intersection of supersets of non-empty open subsets,
also the set
$\Dom(t) = \Dom(t_1)\cap\ldots\cap \Dom(t_s)$
is non-empty; this completes the proof.
\end{proof}

For the remainder of this section,
we write $\LL$ for a field as in Proposition~\ref{prop:fieldchange};
note, however, that the following claims also hold for any 
field~$\LL$.

The next step is to make all coefficients of
a dense triangular system monic.
We will call a triangular system
$(\emptyset, \FFt, k, g)$ 
in $\LL[T_1,\ldots,T_r]$
{\em monic\/}
if $\LC_{k(f)}(f) = 1$ for all $f\in \FFt$.
For instance, the system in Example~\ref{ex:triag}
is monic.

\begin{proposition}[Make monic]
\label{prop:normalize}
Consider a triangular system 
$S := (\emptyset, \FFt, k, g)$
in the ring $\LTT{r}$
that is dense in $\LL[T_n,\ldots,T_r]$
for a $1\leq n\leq r$.
Assume there is $f\in \FFt$ 
with $k(f)=n$
such that $\FF := \FFt\setminus\{f\}$
is monic. 
Then the class $\b h\in R:=\LL[T_{n+1},\ldots,T_r]/\<\mathcal F\>$
of $h:=\LC_{T_n}(f)$ 
is annihilated by a polynomial
 $$
 p
 \ =\ 
 bX^j + fX^{j+1}
 \, \in\, \LL[X]\setminus\{0\}
 \qquad
 \text{with}
 \ 
 f\,\in\, \LL[X],\ b\in \LL^*
 $$
 where 
 $j\in \ZZZ$ is maximal with $X^j\mid p$.
Moreover, writing
$f = hT_n^m+c$
with $m\in \ZZZ$ and $c\in \LL[T_n,\ldots,T_r]$
such that $\deg_{T_n}(c)<m$,
we have a monic dense triangular system $S'$
that is equivalent to~$S$:
$$
S'\ :=\ 
\left(\emptyset, \FF\cup \{f'\}, k, g\right)
\qquad
\text{with}
\ \  
f'
\,:=\, 
T_n^m -\frac{f(h)}{b}c
\ \in\ 
\LL[T_n,\ldots,T_r].
$$
\end{proposition}

\begin{lemma}
\label{lem:ganz} 
\begin{enumerate}

 \item 
 Consider a triangular system 
 $S:=(\emptyset,\FFt,k,g)$ in the ring $\LTT{r}$
 that is dense in $\LL[T_n,\ldots,T_r]$ for a $1\leq n\leq r$.
 Setting $R := \LL[T_n,\ldots,T_r]/\<\FFt\>$,
 the ring extension $\LL\subseteq R$ is integral.
 
 \item 
 Let $\LL\subseteq R$ be a ring extension,
 $I\subseteq R$ an ideal and $h\in R$
 such that $\b{h}\in R/I$ is integral over $\LL$.
 Define $J := \sqrt{I}:h\subseteq R$
 and let 
 $$
 p
 \ =\ 
 bX^j + fX^{j+1}
 \, \in\, \LL[X]
 \qquad
 \text{with}
 \ 
 f\,\in\, \LL[X],\ b\in \LL^*
 $$
 be the minimal polynomial of $\b h$ 
 where
 $j\in \ZZZ$ is maximal with $X^j\mid p$.
 Then
 $h':=-f(h)/b\in R$ yields $hh'-1\in J$. 
\end{enumerate}

\end{lemma}

\begin{proof}
 For (i),
 we write $\FFt = \{f_n,\ldots,f_r\}$ and assume $k(f_i) = i$.
 Define
 $R_j := \LL[T_n,\ldots,T_{r}]/\<f_j,\ldots,f_r\>$
 for $n\leq j\leq r$
 and $R_{r+1} = \LL$.
 The canonical projection 
 $$
 \pi\colon R_{j+1}[T_j]\,\to\, R_j\ =\ R_{j+1}[T_j]/\<\b{f_j}\>,
 \qquad
 f\,\mapsto\,f+\<\b{f_j}\>
 $$
 gives us $\pi(\b{f_j}(T_j))=\b{f_j}(\b{T_j})=\b 0$.
 Since $R_j = R_{j+1}[\b{T_j}]$
 and $\b{f_j}\in R_{j+1}[X]$ is monic, 
 the generator $\b{T_j}$ is integral over $R_{j+1}$
 and non-zero.
 This shows that in the chain
 $
 R = R_n\supseteq \ldots \supseteq R_{r+1}=\LL
 $
 each ring extension is integral, and so is $R\supseteq \LL$.

 We come to (ii).
 Note that $p(h)\in I$
 and $I\subseteq J$ ensures $p(h+J) = 0+J$.
 We have
 \[
  p(\b h)
  \ =\ 
  \left(\b hf(\b h) + b\right)\b h^j
  \ =\ \b 0
  \ \in\ 
  R/J.
 \]
 Observe that $\b h$ is not a zero-divisor:
 for each $x\in R$
 with $xh\in \sqrt{I}:h$, already
 $x\in \sqrt{I}:h$ holds. 
 That is $\b hf(\b h) + b=\b 0$.
 Setting $h' := -f(h)/b$, we obtain $h'h-1\in J$
 from
 \[
  \b{h'h-1}
  \ =\ 
  \b{-\frac{f(h)}{b}h -1}
  \ =\ 
  -\frac{f(\b h)\b h + b}{b}
  \ =\ 
  \b 0\ \ \in\ \ R/J.\qedh
 \]
\end{proof}

\begin{proof}[Proof of Proposition~\ref{prop:normalize}]

Note that the system $(\emptyset,\mathcal F,k,g)$ 
in $\LL[T_{1},\ldots,T_{r}]$
is dense in 
$\LL[T_{n+1},\ldots,T_{r}]$.
By Lemma~\ref{lem:ganz} (i), 
the residue class $\b h\in R$
is integral over $\LL$, i.e., $p$ exists.
Using the inclusion of the ideal
$\sqrt{\<\mathcal F\>}:h\subseteq \LL[T_{n+1},\ldots,T_r]$
in the ideal
$\sqrt{\<\mathcal F\>}:g\subseteq \LTT{r}$,
we obtain
$$
hh'-1\ \in\ \sqrt{\<\mathcal F\>}:h\, \subseteq\,  \sqrt{\<\mathcal F\>}:g
\qquad\text{with }
\ h'\,:=\,\frac{-f(h)}{b}\,\in\, \LL[T_{n+1},\ldots,T_r]
$$
from the second statement of Lemma~\ref{lem:ganz}.
One directly verifies 
the equality of ideals
\[
  \sqrt{\<\FFt\>}:g
  \ =\ 
 \sqrt{\<\FF\>+\<f\>}:g
 \ =\ 
 \sqrt{\<\FF\>+\<f'\>}:g.
\]
In particular, $V(S)=V(S')$ holds with
the dense triangular system
$S'$.
Moreover,
$\LC_1(f')=1$ by choice of $f'$ and
$S'$ is monic.
\end{proof}

In order to make Proposition~\ref{prop:normalize} computational,
we first show how one can compute the required minimal polynomials.

\begin{algorithm}[MinimalPolynomial]
\label{algo:mipo}
{\em Input: } 
an element $g\in R$
where $\LL\subseteq R$ is an integral ring extension
of finite dimension $d:=\dim_\LL(R)$.
\begin{itemize}

\item
Choosing a suitable $\LL$-vector space basis of $R$, 
we consider 
$M := [g^0, \ldots, g^d]$
as a 
$d\times (d+1)$ matrix
over~$\LL$.

\item 
Compute the kernel
$K:=\ker(M)\ne \{0\}$.

\item
Choose $q\in K\subseteq \LL^{d+1}$
such that 
$\max(1\leq j\leq d;\ q_j\ne 0)$
is minimal.

\item  
Define $p_g := q_0X^0 + \ldots + q_dX^d\in \LL[X]$.
\end{itemize}
{\em Output: } 
$p_g\in \LL[X]$.
This is the minimal polynomial  of
$g\in R$.
\end{algorithm}

\begin{proof}
By construction, we have $p(g) = Mq = 0$.
For the minimality, let $p'=\sum_{j=0}^d q_j'X^j\in \LL[X]$ 
be the minimal polynomial of $g$.
Then $Mq' = \sum_{j=0}^d q_j'h^j = p'(h)=0$, i.e., $q'\in K$.
By choice of $q$, we have
$$
\deg(p')
\ =\ 
\max(1\leq j\leq d;\ q_j'\ne 0)
\ \geq\ 
\max(1\leq j\leq d;\ q_j\ne 0)
\ =\ 
\deg(p).\qedh
$$
\end{proof}

\begin{remark}
In Algorithm~\ref{algo:mipo},
the element 
$q\in K$ 
can be computed using Gaussian elimination.
\end{remark}

\begin{algorithm}[Make monic]
\label{algo:normalize}
{\em Input: } 
a triangular system 
$S:=(\emptyset,\FFt,k,g)$ 
that is dense in $\LTT{r}$. 
We assume $\FFt = \{f_1,\ldots,f_r\}$
with $k(f_i)=i$.
\begin{itemize}
\item 
For $n = r$ down to $1$, do:
\begin{itemize}
 \item 
 Set
 $\FFt^n := \{f_i;\ i>n\}
 \subseteq\LL[T_{n+1},\ldots,T_r]$
 and define 
 the dense triangular system
 $(\emptyset,\FFt^n,k,g)$.
  \item
  Decompose $f_n = hT_n^d + c$ 
  with $d\in \ZZ_{\geq 1}$ and 
  $h\in \LL[T_{n+1},\ldots,T_r]$, $c\in \LL[T_n,\ldots,T_r]$ 
  such that $\deg_{T_n}(c)<d$.
  \item 
  Use Algorithm~\ref{algo:mipo} to 
  compute the monic minimal polynomial $p_h\in \LL[X]$
  of $\b h\in\LL[T_{n+1},\ldots,T_r]/\<\FFt^n\>$.
  \item 
  Decompose 
  $p_h= bX^{j+1} + aX^j$ 
  with 
  $b\in \LL[X]$,
  $a\in \LL^*$ by choosing $j\in \ZZZ$ 
  maximal with $X^j\mid p_h$.
  \item 
  Define
  $h' := -b(h)/a$.
  This
  yields
  $hh' - 1 \in \sqrt{\<\FFt^n\>}:h$.
  \item 
  Redefine $f_n$ as $T_n^d+h'c\in \LL[T_n,\ldots,T_r]$.
  Then $S' := (\emptyset,\FFt^n\cup \{f_n\},k,g)$ 
  is a monic triangular system that is dense 
  in $\LL[T_n,\ldots,T_r]$.
\end{itemize}
\end{itemize}
{\em Output: } 
$S'$.
Then $S'$ is a monic triangular system
that is dense in $\LTT{r}$ and is equivalent to~$S$. 
\end{algorithm}

\begin{proof}
Note that the minimal
polynomial $p_h$ exists
 by Lemma~\ref{lem:ganz}~(i)
 since the system is dense.
  By Lemma~\ref{lem:ganz}~(ii),
  $\b h\in \LL[T_{n+1},\ldots,T_r]/\sqrt{\<\FFt^n\>}:h$ 
 is invertible.
The remaining steps are correct by Proposition~\ref{prop:normalize}.
\end{proof}

We now show that 
the existence of solutions 
of a monic, dense triangular system
 can be tested by 
determining a minimal polynomial.

\begin{proposition}[Solvability]
\label{prop:solvable}
Let 
$S := (\emptyset, \FFt, k, g)$
be a monic
triangular system
that is dense in $\LTT{r}$. 
Set
$R := \LTT{r}/\<\FFt\>$.
Then $\LL\subseteq R$ is an integral 
extension and with the minimal polynomial
$p_g\in \LL[X]$ of the residue class $\b g\in R$ 
we have
\[
V(S)\ \ne\ \emptyset
\qquad\gdw\qquad
p_g\ \in \LL[X]\ \text{ is not a monomial}.
\]
\end{proposition}

\begin{lemma}
 \label{lem:solvable}
 In the situation of Proposition~\ref{prop:solvable},
 let $p\in \LL[X]$ be a polynomial
 with $p(g)\in \sqrt{\<\FFt\>}$.
 Then there is $k\in \ZZZ$ such that 
 $p_g\mid p^k$. 
\end{lemma}

\begin{proof}
By assumption, there is $k\in \ZZ_{\ge 1}$ such that 
 $p(g)^k\in \<\FFt\>$, i.e., $p^k(\b g)=\b 0\in R$.
 The monic greatest common denominator 
 $a := \gcd(p^k,p_g)\in \LL[X]$
 satisfies $f(\b g)=\b 0\in R$
 since $p^k(\b g) = p_g(\b g)=\b 0$.
 By minimality of $p_g$, we obtain 
 $p_g = a\mid p^k$.
\end{proof}

\begin{proof}[Proof of Proposition~\ref{prop:solvable}]
Given $x\in V(\FFt)\subseteq\b\LL^r$, the corresponding evaluation homomorphism $\eps_x$ 
fits into the commutative diagram
\[
 \xymatrix{
 \LTT{r}
 \ar[rr]^{\eps_x}
 \ar[dr]_{f\mapsto \b f}
 &&
 \b\LL
 \\
 &
 R
 \ar[ru]_{\phi_x}
 &
 }
\]

The fact, that $\LL\subseteq R$ is integral is Lemma~\ref{lem:ganz}~(i).
Assume now $p_g =X^n$ holds for some $n\in \ZZZ$,
i.e., $\b g\in R$ is nilpotent.
By the diagram, $g(x) = \phi_x(\b g)$ then also is nilpotent
for each $x\in V(\FFt)\subseteq \b\LL^r$.
This means $g(x)=0$.

For the reverse direction, assume $g(x)=0$ holds 
for each $x\in V(\FFt)\subseteq \b\LL^r$, i.e., by the diagram,
 we have $p'(g(x)))=0$ with $p' := X\in \LL[X]$.
By Lemma~\ref{lem:solvable}, there is $k\in \ZZZ$
such that 
$p_g\mid (p')^k=X^k$.
\end{proof}

We now put the previous propositions and algorithms 
together
to obtain an algorithm to check the existence
of solutions of a triangular system.
This completes steps (ii) and (iii) of the list
on page~\ref{list:intro}.

\begin{algorithm}[IsSolvable]
\label{algo:issolvable}
{\em Input: } 
a triangular system $S=(\FFs, \FFt, k, g)$ 
in the ring $\KT{r}$.
\begin{itemize}
\item 
If $\FFs\cap \KK^*$ is non-empty, then:
\begin{itemize}
\item return {\em false\/}.
\end{itemize}
\item 
Consider the triangular system 
$\iota(S)$ 
that is dense in
$\LL[T_{k_1},\ldots,T_{k_s}]$
as in Proposition~\ref{prop:fieldchange}.
\item 
Use Algorithm~\ref{algo:normalize} with input $\iota(S)$
to obtain a monic, dense and equivalent system 
$S'= (\emptyset, \FFt', k', g')$ 
in  $\LL[T_{k_1},\ldots,T_{k_s}]$.
\item 
Use Algorithm~\ref{algo:mipo}
to determine the minimal polynomial $p_{g'}\in \LL[X]$
of the residue class 
$\b{g'}\in \LL[T_{k_1},\ldots,T_{k_s}]/\<\FFt'\>$.
\item
If $p_{g'}$ is a monomial, then:
\begin{itemize}
\item return {\em false\/}.
\end{itemize}
\item return {\em true\/}.
\end{itemize}
{\em Output: } 
{\em true\/} if $V(S)\ne \emptyset$ and {\em false\/} otherwise.
\end{algorithm}

\begin{proof}
By Proposition~\ref{prop:fieldchange}, Algorithm~\ref{algo:mipo}
and Algorithm~\ref{algo:normalize},
$S'$ is equivalent, monic and dense.
Proposition~\ref{prop:solvable} delivers the stated solvability
criterion.
\end{proof}

\section{Monomial containment test and efficiency}
\label{sec:implem}

Putting together steps (i)--(iii) listed on page~\pageref{list:intro},
we are now able to test whether a given ideal $I\subseteq \KT{r}$
contains some monomial $T^\nu$, $\nu \in \ZZZ^r$.
Afterwards, we explore the experimental running time of 
the second author's implementation~\cite{monomtest:implem}
of the algorithm in
\texttt{perl} on a series 
of random polynomials and compare it with Buchberger's algorithm.
Moreover, we compare its efficiency 
on the examples  \texttt{polsys50} from~\cite{eps}
to algorithms listed in~\cite[Tab.~1]{BaeGeLaRo}.

\begin{algorithm}[ContainsMonomial]
\label{algo:containsmonomial}
{\em Input: } 
generators $f_1,\ldots,f_s$
for an ideal $I\subseteq \KT{r}$.
\begin{itemize}
\item 
Define the semi-triangular system 
$S:=(\FFs,\emptyset,0,g)$
where $g := T_1\cdots T_r$,
and $\FFs := \{f_1,\ldots,f_s\}$.
\item 
Let $\SSS$ be the output of Algorithm~\ref{algo:triang}
applied to~$\{S\}$.
\item 
For each $S\in \SSS$, do:
\begin{itemize}
 \item 
 If Algorithm~\ref{algo:issolvable} returns {\em true\/}, then
 \begin{itemize}
  \item Return {\em false\/}.
 \end{itemize}
\end{itemize}
\item 
Return {\em true\/}.
\end{itemize}
{\em Output: }
{\em true\/} if $T^\mu \in I$ for some $\mu \in \ZZZ^r$.
Returns {\em false\/} otherwise.
\end{algorithm}

\begin{remark}
\label{rem:eff}
In the second line of Algorithm~\ref{algo:containsmonomial}
it is more efficient to modify Algorithm~\ref{algo:triang}
such that it checks for solutions immediately after determining
a new semi-triangular system.
\end{remark}

\begin{example}
\label{ex:containsmon}
In the setting of Example~\ref{ex:mush},
we apply Algorithm~\ref{algo:containsmonomial}
with Remark~\ref{rem:eff}
to test whether the ideal $I:=\<f_1,f_2\>\subseteq \KT{4}$
contains a monomial.
To this end, we apply 
Algorithm~\ref{algo:triang}
to the triangle mush $\SSS_0$.
It will first choose the polynomial division
for $(f,h):=(f_1,f_2)$ to obtain
\[
T_4f_1\ =\ 
(T_2-T_3)T_2f_2 - u,
\qquad
u\ =\ 
(T_2^3 - T_3T_2^2)T_4.
\]
This yields a new triangle mush
$\SSS_1 := \{S',S''\}$
where
$S':=
(\{f_2,u\},\emptyset,0,gT_4)$
and $S'':=(\{f_1,f_2,T_4\},\emptyset,0,g)$.
In the next step, we obtain triangle mushes
\begin{eqnarray*}
\SSS_2
&:=&
\left\{
(\{u\},\{f_2\},1,T_4g),\ 
(\{f_1,f_2,T_4\},\emptyset,0,g)
\right\},
\\
\SSS_3
&:=&
\left\{
(\emptyset,\{f_2,u\},4,T_4g),\ 
(\{f_1,f_2,T_4\},\emptyset,0,g)
\right\}.
\end{eqnarray*}

Algorithm~\ref{algo:issolvable} verifies 
that the zero-set $V(f_2,u)\setminus V(T_4g)$ is empty
by the following steps:
first, Algorithm~\ref{algo:normalize}
with input $(\emptyset,\{f_2,u\},4,T_4g)$
will return the monic system
\[
 \left(\emptyset, \{f_2,f_3\}, 4, T_4g\right),
 \qquad 
 f_3
 \ :=\ 
 (T_2-T_3)T_2^2.
\]
As $k(f_2)=1$ and $k(f_3)=2$,
we set $\LL := \KK(T_3,T_4)$
and the ring $R:=\LL[T_1,T_2]/\<f_2,f_3\>$
is integral over $\LL$
with $\LL$-basis $(1, \b{T_2}, \b{T_2}^2)$.
We have
\[
 \b{T_4g}\ =\ \b{(T_3-T_2)T_2T_3T_4}\ \in\ R,
 \qquad
 \b{T_4g}^2\ =\ \b{(T_3-T_2)^2T_2^2T_3^2T_4^2}\ =\ \b 0\ \in\ R,
\]
By Proposition~\ref{prop:solvable},
the algorithm may remove this triangular set, i.e.,
it remains to consider
\begin{eqnarray*}
\SSS_4
&:=&
\left\{
(\{f_1,f_2,T_4\},\emptyset,0,g)
\right\}.
\end{eqnarray*}
The reduction step will remove the redundant equation $f_2$.
The next steps provides us with 
\begin{eqnarray*}
\SSS_5
&:=&
\left\{
(\emptyset,\{f_1,T_4\},4,u'g),\ 
(\{f_1,T_4,u'\},\emptyset,0,g)
\right\},
\qquad
u'\ :=\ 
(T_2-T_3)T_2.
\end{eqnarray*}
By Algorithm~\ref{algo:issolvable},
the system $S:=(\emptyset,\{f_1,T_4\},4,u'g)$
has a solution:
similar to before, 
Algorithm~\ref{algo:normalize}
returns the monic system
\[
 \left(\emptyset,\{f_4,T_4\},4,u'g\right),
 \qquad 
 f_4
 \ :=\ 
 T_1-T_3
\]
with $k(f_4)=1$ and $k(T_4)=4$.
Setting $\LL:=\KK(T_2,T_3)$, the ring extension
$\LL\subseteq R:=\LL[T_1,T_4]/\<f_4,T_4\>$ is integral
with $\LL$-basis $(1)$.
Since
\[
 \b{u'g}\ =\ \b{(T_2-T_3)T_1T_2^2T_3}\ =\ \b{(T_2-T_3)T_2^2T_3^2}\ \in\ R
\]
is non-zero,
its minimal polynomial 
$p = X - (T_2-T_3)T_2^2T_3^2\in \LL[X]$
is not a monomial, i.e.,
$V(S)\ne \emptyset$ by Proposition~\ref{prop:solvable}.
Thus, $\SSS_0$ has a solution as we already witnessed 
in Example~\ref{ex:mush}.
In particular, $I$ contains no monomial, i.e.,
the algorithm returns {\em false\/}.
\end{example}

The remainder of this note is devoted to experimental running times.
We apply the \texttt{perl} implementation~\cite{monomtest:implem} of 
Algorithm~\ref{algo:containsmonomial}
to a series of random
ideals $\<f_1,\ldots,f_s\>\subseteq \KT{r}$ for fixed $2\leq s\leq 5$
and running $1\leq r\leq 10$.
Moreover, 
setting $\FF := \{f_1,\ldots,f_s\}$,
we distinguish the cases 
$V(\FF)=\emptyset$ and $V(\FF)\ne \emptyset$.

To make the experimental running times better comparable
to Buchberger's Gr\"obner basis algorithm~\cite{CoLiOSh},
we have reimplemented the latter in \texttt{perl}
in two variants: the first one is the classical version
whereas the second one stops as soon as a monomial could be found.
Both algorithms as well as the testing sets $\FF$ 
are available at~\cite{monomtest:implem}.
The following graphics show the averages over the successful tests.

\begin{center}
 \begin{minipage}{6cm}
 \footnotesize
 \begin{center}
 \includegraphics[width=5.5cm]{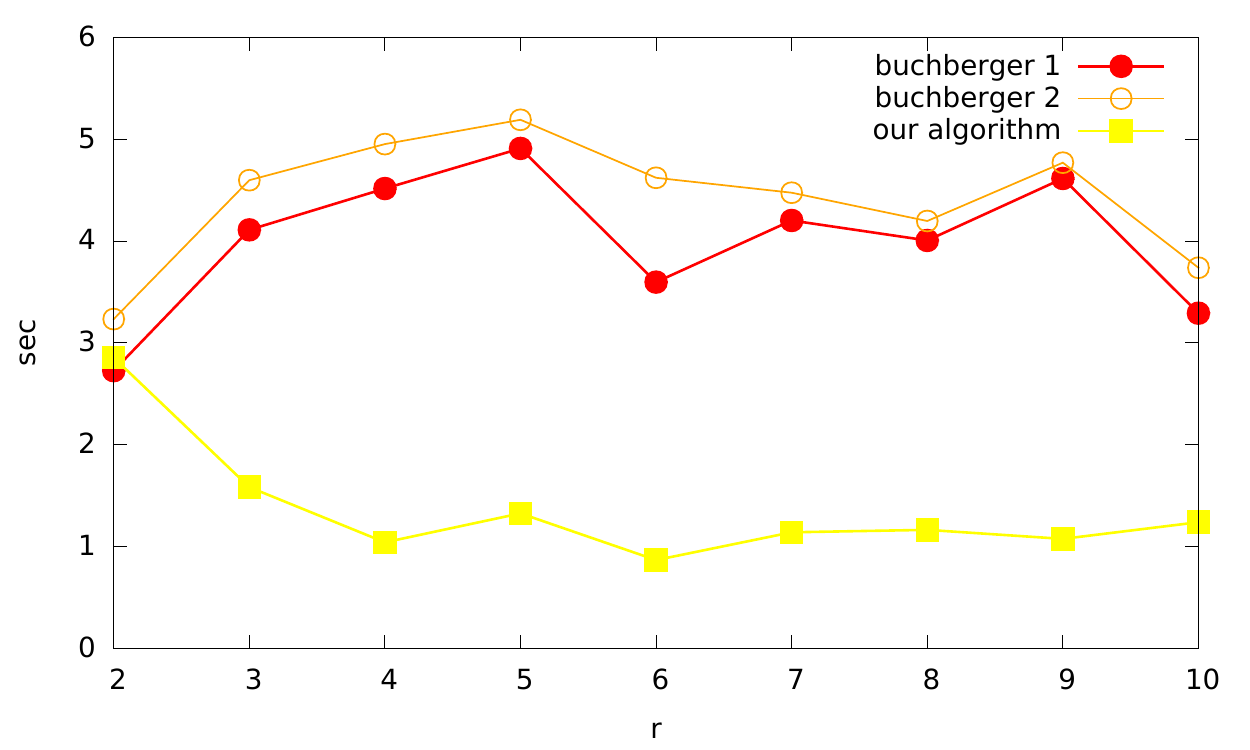}
%
%
%
%
 $V(\FF)\ne\emptyset$, \quad $s=2$
 \end{center}
 \end{minipage}
\ \  
\begin{minipage}{6cm}
\footnotesize
\begin{center}
\includegraphics[width=5.5cm]{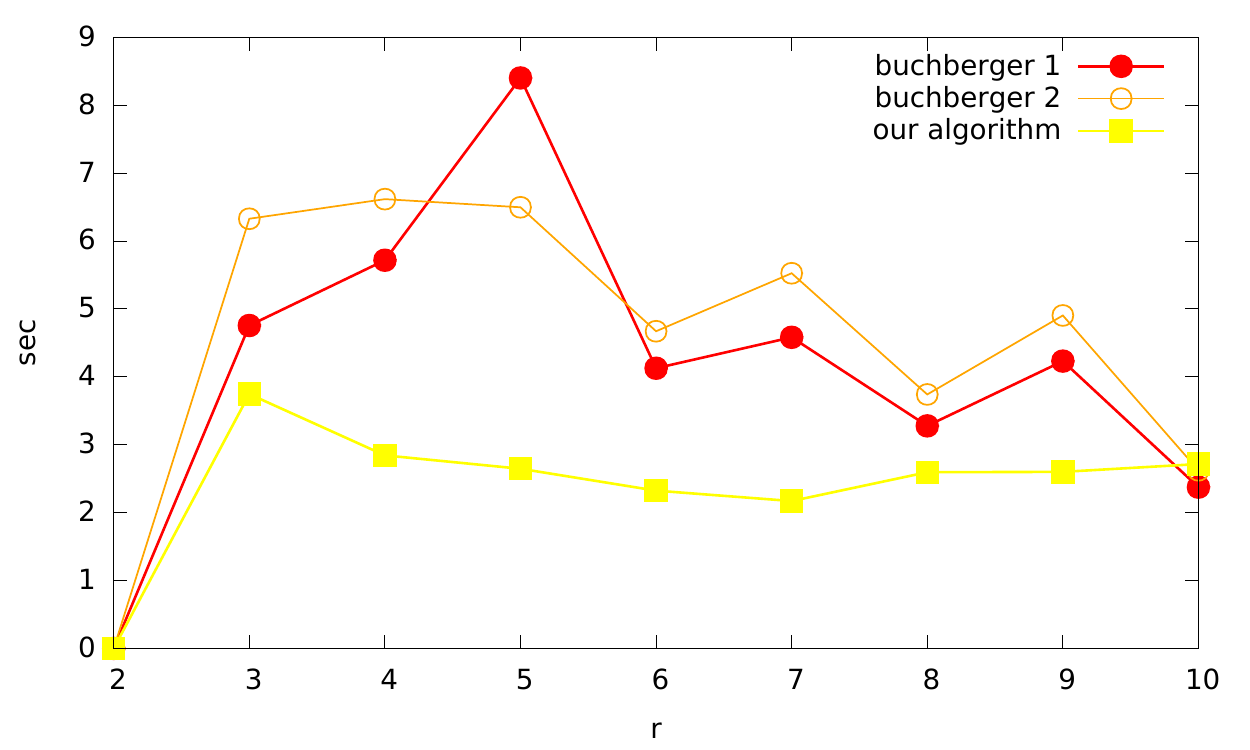}
%
%
%
%
 $V(\FF)\ne\emptyset$, \quad $s=3$
 \end{center}
\end{minipage}
\\ 
\begin{minipage}{6cm}
\footnotesize
\begin{center}
\includegraphics[width=5.5cm]{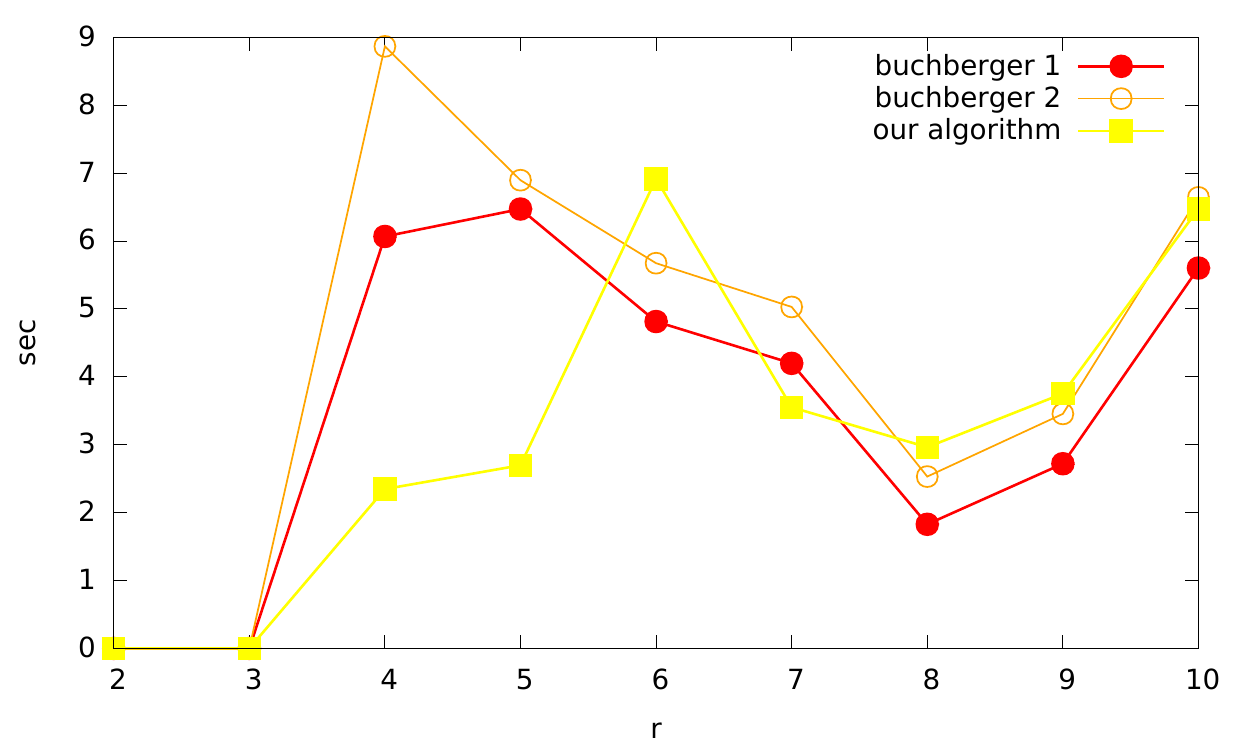}
%
%
%
%
 $V(\FF)\ne\emptyset$, \quad $s=4$
 \end{center}
\end{minipage}
\ \ 
\begin{minipage}{6cm}
\footnotesize
\begin{center}
\includegraphics[width=5.5cm]{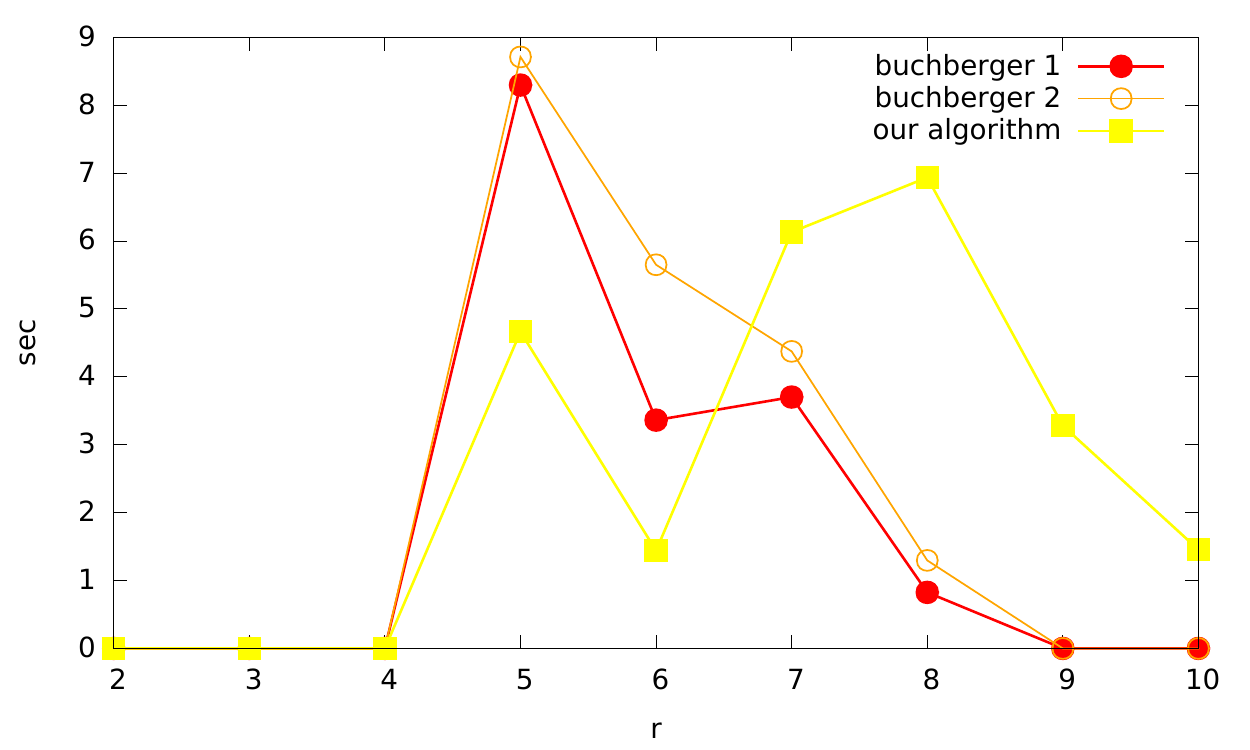}
%
%
%
%
 $V(\FF)\ne\emptyset$, \quad $s=5$
 \end{center}
\end{minipage}
\\

 \begin{minipage}{6cm}
 \footnotesize
 \begin{center}
 \includegraphics[width=5.5cm]{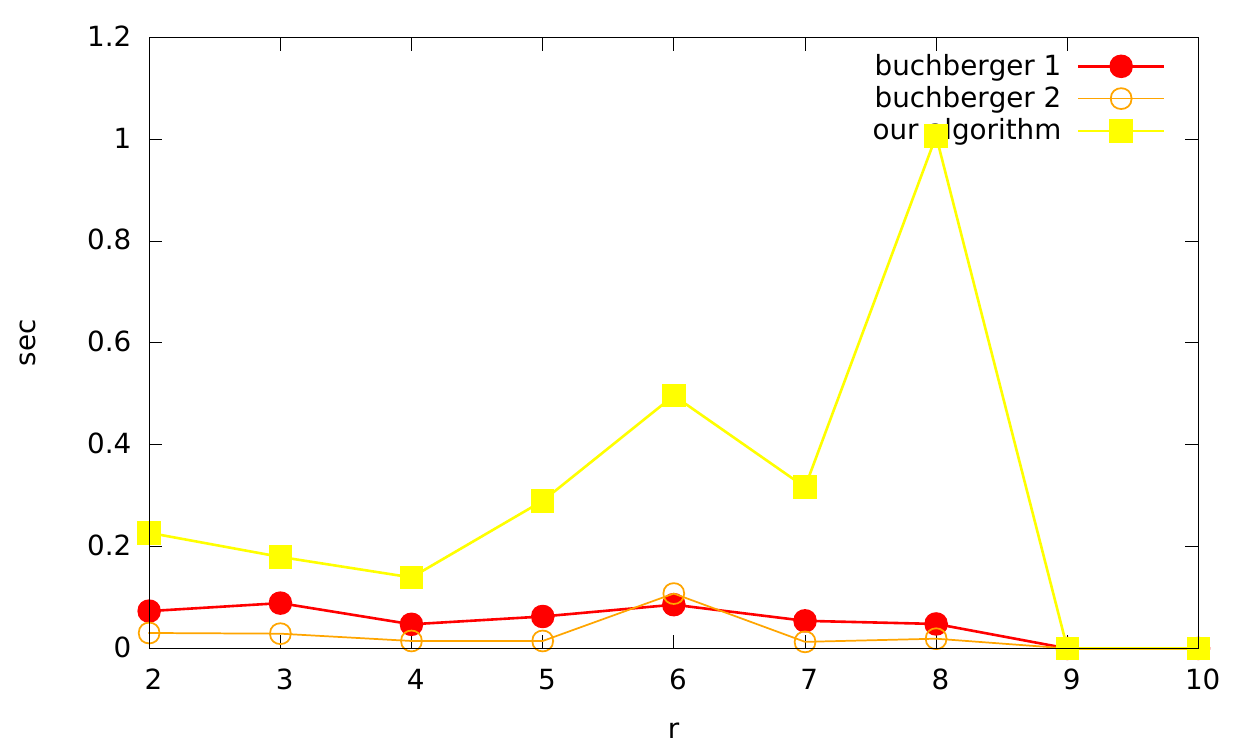}
%
%
%
%
 $V(\FF)=\emptyset$, \quad $s=2$
 \end{center}
 \end{minipage}
\ \  
\begin{minipage}{6cm}
\footnotesize
\begin{center}
\includegraphics[width=5.5cm]{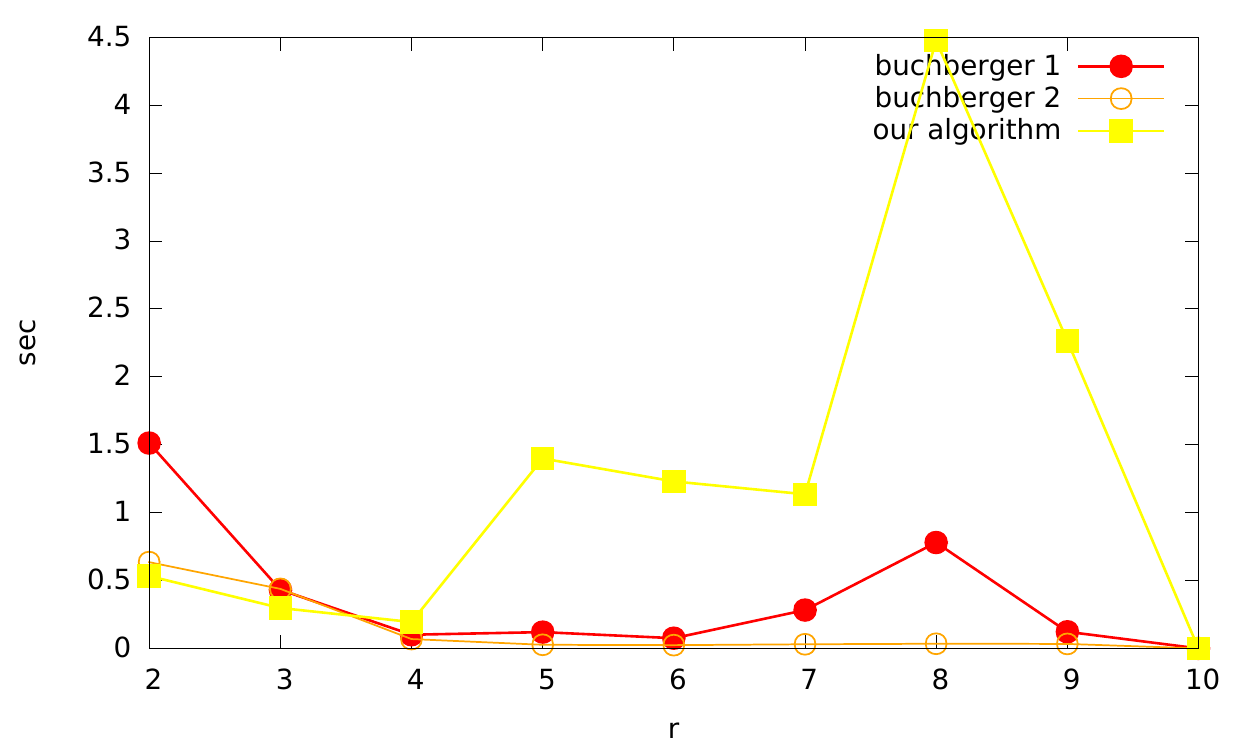}
%
%
%
%
 $V(\FF)=\emptyset$, \quad $s=3$
 \end{center}
\end{minipage}
\\ 
\begin{minipage}{6cm}
\footnotesize
\begin{center}
\includegraphics[width=5.5cm]{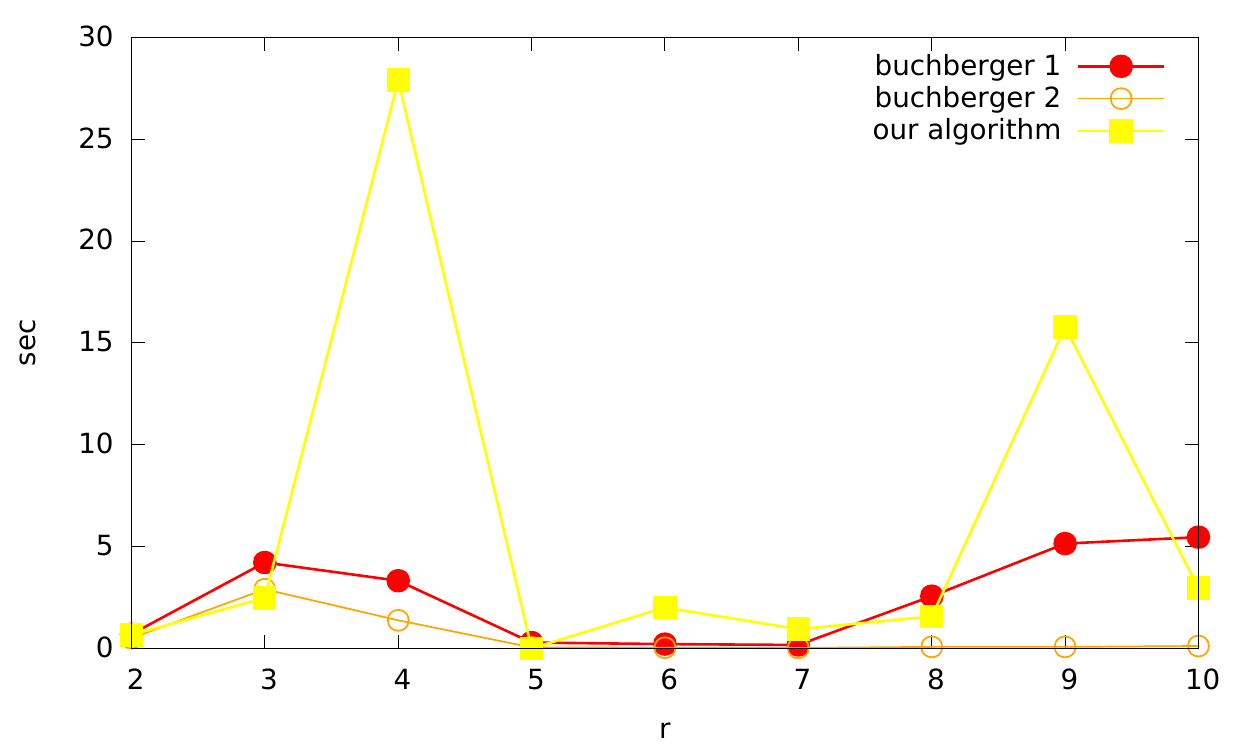}
%
%
%
 $V(\FF)=\emptyset$, \quad $s=4$
 \end{center}
\end{minipage}
\ \ 
\begin{minipage}{6cm}
\footnotesize
\begin{center}
\includegraphics[width=5.5cm]{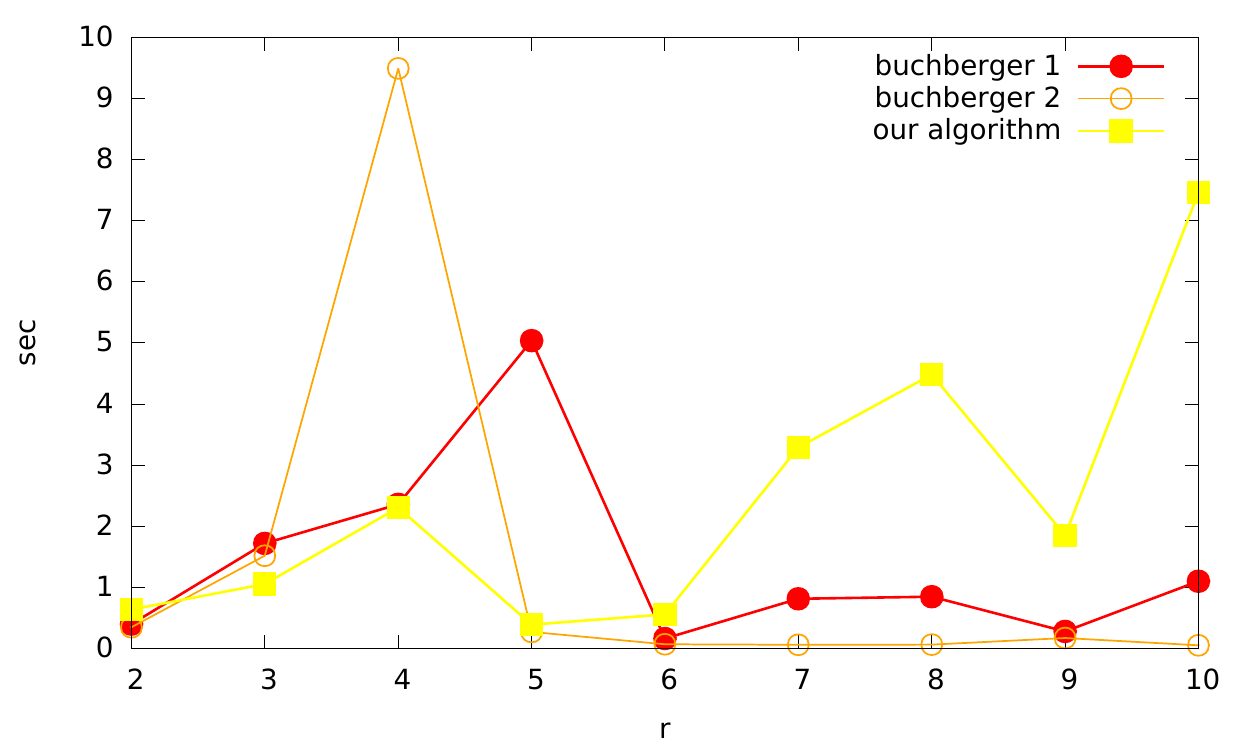}
%
%
%
%
 $V(\FF)=\emptyset$, \quad $s=5$
 \end{center}
\end{minipage}
\end{center}

On the given set of polynomials, 
Algorithm~\ref{algo:containsmonomial}
seems to be competitive when $V(\FF)\ne \emptyset$
whereas, for $V(\FF)=\emptyset$, the classical
Buchberger's algorithm usually needs less time.

Additionally, we have applied Algorithm~\ref{algo:containsmonomial}
to the set of examples \texttt{polsys50} from~\cite{eps};
its running time as well as the number of performed additions on 
a 2.66 GHz machine with time bound $300$ seconds and at most 1 GB of RAM
is listed in the left-hand side part of the following table.
We write ``n/a'' if the computation was unsuccessful either due to time
reasons or because it was out of memory.

Moreover, in the right-hand part of the table, we list some of the running 
times listed in~\cite[Table~1]{BaeGeLaRo} on the same examples.
We want to stress the fact that 
the two sides of this table are 
only marginally comparable: 
not only is the goal different (\cite{BaeGeLaRo} deduces more information
on the solutions whereas we test the existence of solutions),
also the machines and maximal running times / memory are different.

\begingroup
\tiny
\begin{longtable}{rrcrc|rrr}
\myhline
no.
&  
time~\ref{algo:containsmonomial}
&
result
&
add.s
&
&
time RC1
&
time DW1
&
time AT1
\\
\myhline

1 & $>$ 300 & n/a & n/a & & 3.5 & 0.4 & 3.0 \\ \hline
2 & $>$ 300 & n/a & n/a & & 7.4 & 7.6 & 7.1 \\ \hline
3 & 30.89 & 1 & 4956 & & $>$ 3h & 985.7 & 7538.0\\ \hline
4 & $>$ 300 & n/a & n/a & & $>$ 4 GB & $>$ 4 GB & 0.2 \\ \hline
5 & 0.62 & 0 & 2449 & &  &\\ \hline
6 & 2.25 & 1 & 4239 & & 0.4 & 0.1 & 0.2\\ \hline
7 & $>$ 1 GB & n/a & n/a &  & $>$ 3h & 7352.6 & $>$ 4 GB \\ \hline
8 & 0.14 & 1 & 214 & & &\\ \hline
9 & 11.75 & 1 & 10149 & & &\\ \hline
10 & 0.21 & 1 & 517 & & &\\ \hline
11 & 0.17 & 1 & 361 & & &\\ \hline
12 & 0.74 & 1 & 1909 & & 0.5 & 0.3 & 0.4\\ \hline
13 & 0.15 & 1 & 214 & &  &\\ \hline
14 & 0.23 & 1 & 442 & & 0.5 & $>$ 3h & 1.5\\ \hline
15 & 29.82 & 1 & 6655 & & &\\ \hline
16 & $>$ 300 & n/a & n/a &  & 0.9 & 1.4 & 1.8 \\ \hline
17 & $>$ 300 & n/a & n/a &  & 6.5 & 4.7 & 75.5 \\ \hline
18 & 2.34 & 1 & 4324 & & 0.3 & 0.1 & 0.1\\ \hline
19 & $>$ 300 & n/a & n/a &  & 419.9 & 0.4 & 0.4 \\ \hline
20 & 0.25 & 1 & 668 & & &\\ \hline
21 & $>$ 300 & n/a & n/a &  & 1.6 & 86.6 & 4.5\\ \hline
22 & $>$ 300 & n/a & n/a &  & 0.6 & 1.2 & 1.5 \\ \hline
23 & $>$ 300 & n/a & n/a &  & 0.4 & 0.1 & 29.5 \\ \hline
24 & $>$ 300 & n/a & n/a &  & 1.2 & 1.3 & 1.0 \\ \hline
25 & 0.25 & 1 & 537 & & 1.2 & $>$ 3h & $>$ 4 GB\\ \hline
26 & 1.07 & 0 & 4610 & & &\\ \hline
27 & 1.47 & 1 & 2320 & & &\\ \hline
28 & 9.89 & 1 & 6632 & & &\\ \hline
29 & 0.15 & 1 & 297 & & 0,3 & 0,3 & 0,3\\ \hline
30 & $>$ 300 & n/a & n/a &  & $>$ 4 GB & $>$ 4 GB & 45.3 \\ \hline
31 & $>$ 1 GB & n/a & n/a &  & $>$ 4 GB & $>$ 4 GB & $>$ 3h \\ \hline
32 & 0.41 & 1 & 1200 & & &\\ \hline
33 & $>$ 1 GB & n/a & n/a &  & 3.4 & 1.3 & 3.5 \\ \hline
34 & $>$ 300 & n/a & n/a &  & 911.5 & $>$ 3h & $>$ 4 GB \\ \hline
35 & $>$ 300 & n/a & n/a &  & 1.5 & 1.2 & 1.7 \\ \hline
36 & 0.13 & 1 & 160 & & &\\ \hline
37 & 0.27 & 1 & 633 & & &\\ \hline
38 & 11.10 & 1 & 6878 & & &\\ \hline
39 & $>$ 300 & n/a & n/a &  & 0.6 & 1.2 & 0.6 \\ \hline
40 & $>$ 300 & n/a & n/a &  &  & \\ \hline
41 & $>$ 300 & n/a & n/a &  & 1.5 & 1.5 & 7.0 \\ \hline
42 & 0.35 & 1 & 1028 & & &\\ \hline
43 & $>$ 1 GB & n/a & n/a &  & 0.7 & 3.1 & 0.2 \\ \hline
44 & $>$ 300 & n/a & n/a &  & 24.5 & 3.4 & 1.2 \\ \hline
45 & $>$ 300 & n/a & n/a &  & & \\ \hline
46 & $>$ 300 & n/a & n/a &  & & \\ \hline
47 & 16.40 & 1 & 10465 & & 1.3 & 2.8 & 13.0\\ \hline
48 & 0.23 & 1 & 563 & & &\\ \hline
49 & $>$ 300 & n/a & n/a & & 0.3 & 610.2 & 0.5 
\\
\myhline
\end{longtable}
\endgroup

\end{document}